\documentclass[11pt]{amsart}
\usepackage{mathrsfs}
\usepackage{amsmath,amssymb}
\usepackage{latexsym,amsfonts,amssymb}

\newtheorem{theorem}{Theorem}[section]
\newtheorem{lemma}[theorem]{Lemma}
\newtheorem{corollary}[theorem]{Corollary}
\newtheorem{question}[theorem]{Question}
\theoremstyle{definition}
\newtheorem{definition}[theorem]{Definition}
\newtheorem{proposition}[theorem]{Proposition}
\theoremstyle{remark}

\newtheorem{example}[theorem]{Example}

\begin{document}

\title[$\pi$-metrizable spaces and strongly $\pi$-metrizable spaces]
{$\pi$-metrizable spaces and strongly $\pi$-metrizable spaces}

\author{Fucai Lin}%
\address{Fucai Lin: Department of Mathematics and Information Science, Zhangzhou Normal University,
Fujian, Zhangzhou, 363000, P.R.China}%
\email{linfucai2008@yahoo.com.cn}
\author{Shou Lin}
\address{(Shou Lin): Department of Mathematics,
Zhangzhou Normal University, Zhangzhou 363000, P. R. China;
Institute of Mathematics, Ningde Teachers' College, Ningde, Fujian
352100, P. R. China} \email{linshou@public.ndptt.fj.cn}
\thanks{Supported by the NSFC (No. 10971185), the NSF of Fujian Province (No. 2009J01013)
and the Educational Department of Fujian Province (No. JA09166) of
China.}

\subjclass[2000]{54B10, 54C10, 54D70}%
\keywords{$\pi$-metrizable spaces; locally $\pi$-metrizable;
strongly $\pi$-metrizable spaces; perfect maps.}

\begin{abstract}
A space $X$ is said to be $\pi$-metrizable if it has a
$\sigma$-discrete $\pi$-base. In this paper, we mainly give
affirmative answers for two questions about $\pi$-metrizable spaces.
The main results are that: (1) A space $X$ is $\pi$-metrizable if
and only if $X$ has a $\sigma$-hereditarily closure-preserving
$\pi$-base; (2) $X$ is $\pi$-metrizable if and only if $X$ is almost
$\sigma$-paracompact and locally $\pi$-metrizable; (3) Open and
closed maps preserve $\pi$-metrizability; (4) $\pi$-metrizability
satisfies hereditarily closure-preserving regular closed sum
theorems. Moreover, we define the notions of second-countable
$\pi$-metrizable and strongly $\pi$-metrizable spaces, and study
some related questions. Some questions about strongly
$\pi$-metrizability are posed.
\end{abstract}
\maketitle

\section{Introduction}
$\pi$-metrizable spaces were first studied by V. Ponomarev as a
necessary conditions for being the absolute of a metrizable
space~\cite{PV}. In \cite{FD}, D. Fearnley has constructed a Moore
and $\pi$-metrizable space which cannot be densely embedded in any
Moore space with the Baire property. In \cite{sd}, D. Stover has
proved that a space $X$ is $\pi$-metrizable if and only if $X$ has a
$\sigma$-locally finite $\pi$-base. It is well known that a regular
space is metrizable if and only if it has a $\sigma$-hereditarily
closure-preserving base. Recently, C. Liu posed the following two
questions in a private communication with the authors.

\begin{question}\label{q0}
If $X$ has a $\sigma$-hereditarily closure-preserving $\pi$-base, is
$X$ $\pi$-metrizable?
\end{question}

\begin{question}\label{q1}
Is $\pi$-metrizability preserved by open and closed maps?
\end{question}

Obviously, if the Question~\ref{q0} is affirmative, then
Question~\ref{q1} is also affirmative.

In this paper, we shall give an affirmative answer for
Questions~\ref{q0} and \ref{q1}, respectively. In fact, we prove
that quasi-open and closed maps preserve $\pi$-metrizability. We
also improve some results in \cite{sd}. Moreover, we define the
notions of second-countable $\pi$-metrizable and strongly
$\pi$-metrizable spaces, and study some related questions.

\begin{definition}
Let $X$ be a space. A collection of nonempty open sets $\mathcal{U}$
of $X$ is called a {\it $\pi$-base} if for every nonempty open set
$O$, there exists an $U\in\mathcal{U}$ such that $U\subset O$. A
space $X$ is said to be {\it $\pi$-metrizable} if it has a
$\sigma$-discrete $\pi$-base. A space $X$ is called a {\it
second-countable $\pi$-metrizable space} if $X$ has a countable
$\pi$-base.
\end{definition}

Obviously, every second-countable $\pi$-metrizable space is
$\pi$-metrizable.

\begin{definition} Let $f:X\rightarrow
Y$ be a map.
\begin{enumerate}
\item $f$ is a {\it compact map} if each $f^{-1}(y)$ is compact in $X$;

\item $f$ is a {\it perfect map} if it is a closed and compact map;

\item $f$ is a {\it quasi-open map} if $\mbox{Int}f(U)\neq\emptyset$ for any non-empty open subset $U$ of $X$;

\item $f$ is called  {\it at most $k$-to-one map} if $|f^{-1}(y)|\leq k$ for every $y\in Y$, where $k\in
\mathbb{N}$;

\item $f$ is an {\it irreducible map} if there does not exist a
proper closed subset $X'$ of $X$ such that $f(X')=Y$.
\end{enumerate}
\end{definition}

\begin{definition}\cite{BD}
Let $\mathcal{P}$ be a family of subsets of a space $X$.
$\mathcal{P}$ is {\it hereditarily closure-preserving} (abbrev. HCP)
if, whenever a subset $S(P)\subset P$ is chosen for each
$P\in\mathcal{P}$, the family $\{S(P): P\in \mathcal{P}\}$ is
closure-preserving.
\end{definition}

\begin{definition}\cite{sd}
A space $X$ is called {\it strongly d-separable} if there exists
$\{K_{n}: n\in \mathbb{N}\}$ such that each $K_{n}$ is a closed
discrete subset of $X$ and $\cup\{K_{n}: n\in \mathbb{N}\}$ is dense
in $X$.
\end{definition}

For a topological space $X$, let $\mathcal{P}$ be a family of
subsets of $X$, and let
$$I(X)=\{x:x \mbox{ is an isolated point of } X\},$$
$$(\mathcal{P})_{x}=\{P\in\mathcal{P}: x\in P\}\ \mbox{for each}\ x\in X.$$
However, we denote $\mathcal{P}_{x}$ by a subfamily of $(\mathcal{P})_{x}$
for each $x\in X$.

Throughout this paper, all spaces are assumed to be $T_{1}$ and
regular, all maps are continuous and onto. Denote the positive natural
numbers by $\mathbb{N}$. We refer the reader to \cite{E} for
notations and terminology not explicitly given here.

\medskip
\section{$\pi$-metrizable spaces}

First, we give two technical lemmas in order to give an
affirmative answer for Question~\ref{q0}.

\begin{lemma}\cite{BD}\label{l0}
Let $\mathcal{P}$ be a HCP collection of open subsets of $X$ and
$A\subset X$. If $x\in A^{d}$ and $G$ is a $G_{\delta}$-set of $X$
such that $x\in G$ and $G\cap (A-\{x\})=\emptyset$, then
$(\mathcal{P})_{x}$ is finite
\end{lemma}

\begin{lemma}\label{l1}
Let $X$ have a $\sigma$-HCP $\pi$-base $\mathcal{P}=\bigcup_{n\in
\mathbb{N}}\mathcal{P}_{n}$, where $\mathcal{P}_{n}$ is HCP for each
$n\in \mathbb{N}$. Then $(\mathcal{P}_{n})_{x}$ is finite for each
$x\in X\setminus I(X)$ and $n\in \mathbb{N}$.
\end{lemma}

\begin{proof}
Fix a point $x\in X\setminus I(X)$. For each $n\in \mathbb{N}$ and
$P\in \mathcal{P}_{n}$, we choose a point $x_{P}\in
P\setminus\{x\}$. Let $F_{n}=\{x_{P}:P\in \mathcal{P}_{n}\}$. Then
$F_{n}$ is closed. Put $A=\bigcup_{n\in \mathbb{N}}F_{n}$ and
$G=X-A$. For each $x\in U$ with $U$ open in $X$, there exists a
$P\in \mathcal{P}$ such that $P\subset U$, and hence $x_{P}\in U\cap
(A-\{x\})\neq\emptyset$. Therefore, $x\in A^{d}\cap G$. Obviously,
$G$ is a $G_{\delta}$-set and $G\cap (A-\{x\})=\emptyset$. Hence
$(\mathcal{P}_{n})_{x}$ is finite by Lemma~\ref{l0}.
\end{proof}

A collection of sets $\mathcal{U}$ in a space $X$ each with nonempty
interior is called a {\it $\pi_{\ast}$-base}~\cite{sd} if for each
open set $O$ there is an $U\in\mathcal{U}$ such that $U\subset O$.

\begin{theorem}\label{t0}
For a topological space $X$, the following are equivalent:
\begin{enumerate}
\item $X$ is a
$\pi$-metrizable space;
\item $X$ has a $\sigma$-HCP $\pi$-base;
\item $X$ has a $\sigma$-locally finite $\pi$-base.
\end{enumerate}
\end{theorem}

\begin{proof}
(1)$\Rightarrow$(2) is obvious. (3)$\Rightarrow$(1) by \cite[Theorem
2.2]{sd}. Hence we only need to prove (2)$\Rightarrow$(3).

Let $\mathcal{P}=\bigcup_{n\in \mathbb{N}}\mathcal{P}_{n}$ be a
$\sigma$-HCP $\pi$-base of $X$, where each $\mathcal{P}_{n}$ is HCP.
By the regularity, for each $P\in \mathcal{P}$, there is a nonempty
closed subset $B_{P}$ in $X$ such that $B_{P}\subset P$, and if
$P\not\subset I(X)$ then $\mbox{int}(B_{P})\cap (X\setminus
I(X))\neq\emptyset$. Let $\mathcal{B}=\bigcup_{n\in
\mathbb{N}}\mathcal{B}_{n}$, where $\mathcal{B}_{n}=\{B_{P}: P\in
\mathcal{P}_{n}\}$ for each $n\in \mathbb{N}$. It is easy to see
that $\mathcal{B}=\bigcup_{n\in \mathbb{N}}\mathcal{B}_{n}$ is a
$\sigma$-HCP $\pi_{\ast}$-base of $X$. For each $n\in \mathbb{N}$,
let $X(n)=\{x\in X: \mathcal{B}_{n}\ \mbox{is locally finite at
point}\ x\}$.

Claim: For each $n\in \mathbb{N}$, $X\setminus I(X)\subset X(n)$.

Indeed, put $x\in X\setminus I(X)$. It follows from Lemma~\ref{l1} that
$(\mathcal{P}_{n})_{x}$ is finite, and hence $(\mathcal{B}_{n})_{x}$
is also finite. Therefore, $\cup (\mathcal{B}_{n}\setminus
(\mathcal{B}_{n})_{x})$ is closed and does not contain $x$, and
hence $X\setminus (\cup (\mathcal{B}_{n}\setminus
(\mathcal{B}_{n})_{x}))$ is an open neighborhood of $x$ and at most
intersects finitely many elements of $\mathcal{B}_{n}$. So, $x\in
X(n)$.

It is obvious that $X(n)$ is open for each $n\in \mathbb{N}$. Let
$\mathcal{P}_{n}^{\prime}=\{\mbox{int}(B)\cap X(n):
B\in\mathcal{B}_{n}\}$. Then $\mathcal{P}_{n}^{\prime}$ is a locally
finite collection of open subsets of $X$ for each $n\in \mathbb{N}$.
Put $\mathcal{P}_{n}^{\prime\prime}=\{\{x\}: \{x\}\in
\mathcal{B}_{n}\}$ for each $n\in \mathbb{N}$. Then
$\mathcal{P}_{n}^{\prime\prime}$ is discrete for each $n\in
\mathbb{N}$. Let $\mathcal{P}^{\prime}=\bigcup_{n\in
\mathbb{N}}(\mathcal{P}_{n}^{\prime}\cup\mathcal{P}_{n}^{\prime\prime})$.
It is easy to see that $\mathcal{P}^{\prime}$ is a  $\sigma$-locally
finite $\pi$-base for $X$. Indeed, for each nonempty open subset $O$
of $X$, if $O\cap I(X)\neq\emptyset$, then we choose a point $x\in
O\cap I(X)$ and therefore, $\{x\}\in\mathcal{B}$ and $\{x\}\subset O$;
if $O\cap I(X)=\emptyset$, then there is a $B\in \mathcal{B}$ with
$B\subset O$ since $\mathcal{B}$ is a $\pi_{\ast}$-base, and
therefore, $\emptyset\neq\mbox{int}(B)\cap X(n)\subset O$ by the
Claim.
\end{proof}

\begin{corollary}\label{c1}
A space $X$ is $\pi$-metrizable if and only if $X$ has a
$\sigma$-HCP $\pi_{\ast}$-base $\mathcal{P}$ such that, for every
$P\in\mathcal{P}$, $P$ is a regular closed set of $X$.
\end{corollary}

In \cite{sd}, D. Stover has proved that open perfect or irreducible
perfect maps preserve $\pi$-metrizability. However, we have the
following Theorem~\ref{t11}, Corollaries~\ref{c2} and~\ref{c3},
which give an affirmative answer for Question~\ref{q1} and also
improve some results in \cite{sd}.

\begin{theorem}\label{t11}
Quasi-open and closed maps preserve $\pi$-metrizability.
\end{theorem}

\begin{proof}
Let $f:X\rightarrow Y$ be an quasi-open and closed map, where $X$ is
a $\pi$-metrizable space. It follows from Theorem~\ref{t0} that $X$
has a $\sigma$-HCP $\pi$-base $\mathcal{P}$. Since closed maps
preserve HCP collections, $f(\mathcal{P})$ is a $\sigma$-HCP
collection of subsets of $Y$. Since $f$ is a quasi-open map,
$\{\mbox{int}f(P): P\in \mathcal{P}\}$ is a $\pi$-base for $X$.
Hence $Y$ is a $\pi$-metrizable space by Theorem~\ref{t0}.
\end{proof}

\begin{corollary}\label{c2}
Open and closed maps preserve $\pi$-metrizability.
\end{corollary}

\begin{corollary}\label{c3}
Irreducible closed maps preserve $\pi$-metrizability.
\end{corollary}

\begin{proof}
It follows from the definition of the irreducible closed mappings that an irreducible closed map is
quasi-open. Therefore, irreducible closed maps preserve
$\pi$-metrizability by Theorem~\ref{t11}.
\end{proof}

However, perfect maps don't preserve $\pi$-metrizability, see
Example~\ref{e0}.

A topological property $\mathscr{P}$ satisfies {\it hereditarily
closure-preserving regular closed sum theorems} if a topological
space $X$ has a hereditarily closure-preserving regular closed
covering $\{F_{\alpha}\}_{\alpha\in A}$ such that $F_{\alpha}$ has
topological property $\mathscr{P}$ for every $\alpha\in A$, then $X$
has topological property $\mathscr{P}$.

\begin{lemma}\label{l2}
Suppose the topological property $\mathscr{P}$ satisfies the following two
conditions:
\begin{enumerate}
\item $\mathscr{P}$ is preserved by
topological sums;

\item $\mathscr{P}$ is preserved by
quasi-open and closed maps,
\end{enumerate}
then $\mathscr{P}$ satisfies hereditarily closure-preserving regular
closed sum theorem.
\end{lemma}

\begin{proof}
Let $\{F_{\alpha}\}_{\alpha\in A}$ be a hereditarily
closure-preserving regular closed covering for a space $X$, where
$F_{\alpha}$ has topological property $\mathscr{P}$ for every
$\alpha\in A$. For every $\alpha\in A$, let $F_{\alpha}^{\prime}$
denote a copy of $F_{\alpha}$ and let $f_{\alpha}$ be this
homeomorphism. Put $X^{\ast}$ be the disjoint topological sum of
$F_{\alpha}^{\prime}$, and define a map $f$ from $X^{\ast}$ onto $X$
as follows: for every $x\in X^{\ast}$, if $x\in
F_{\alpha}^{\prime}$, then $f(x)=f_{\alpha}(x)$.

Obviously, $f$ is a map. It follows from (1) that $X^{\ast}$ has
topological property $\mathscr{P}$. It is easy to see that $f$ is a
closed map since $\{F_{\alpha}\}_{\alpha\in A}$ is HCP. Now we only
need to show that $f$ is a quasi-open map. Since $F_{\alpha}$ is a
regular closed set, there is an open subset $U_{\alpha}$ of $X$ such
that $F_{\alpha}=\overline{U_{\alpha}}$. Clearly, it is sufficient
to show that $\mbox{int}f(E)\neq\emptyset$ for each non-empty open
subset $E$ in $F_{\alpha}^{\prime}$. Since $f_{\alpha}:
F_{\alpha}^{\prime}\rightarrow \overline{U_{\alpha}}$ is a
homeomorphism map, $f_{\alpha}(E)$ is open in
$\overline{U_{\alpha}}$, and therefore, there exists an open subset
$U$ in $X$ such that $f_{\alpha}(E)=U\cap\overline{U_{\alpha}}$.
Choose a point $x\in f_{\alpha}(E)$. Then there is an open subset
$V(x)$ of $X$ such that $x\in V(x)\subset U$. Since $x\in
f_{\alpha}(E)\subset \overline{U_{\alpha}}$, $V(x)\cap
U_{\alpha}\neq\emptyset$. Then $V(x)\cap U_{\alpha}\subset
f_{\alpha}(E)$, and hence $\mbox{int}f_{\alpha}(E)\neq\emptyset$.
Since $E\subset F_{\alpha}^{\prime}$, $f(E)=f_{\alpha}(E)$. Then $f$
is quasi-open. Therefore, $X$ has topological property $\mathscr{P}$
by (2).
\end{proof}

\begin{theorem}\label{t12}
$\pi$-metrizability satisfies hereditarily closure-preserving
regular closed sum theorems.
\end{theorem}

\begin{proof}
It is easy to prove that $\pi$-metrizability is preserved by
topological sums. Since $\pi$-metrizability is preserved by
quasi-open and closed maps, $\pi$-metrizability satisfies locally
finite regular closed sum theorem by Lemma~\ref{l2}.
\end{proof}

It is well known that a space $X$ is metrizable if and only if $X$
is paracompact and locally metrizable. However, there exists a
$\pi$-metrizable space such that $X$ is non-paracompact. But we have
the following Theorem~\ref{t10}.

A space $X$ is called {\it almost $\sigma$-paracompact} if, for each
open covering $\mathcal{U}$ of $X$, there is a $\sigma$-locally
finite open collection $\mathcal{V}$ such that $\mathcal{V}$ refines
$\mathcal{U}$ and $\cup\mathcal{V}$ is dense in $X$. Obviously,
paracompact or $\pi$-metrizable spaces are almost
$\sigma$-paracompact.

\begin{theorem}\label{t10}
A space $X$ is $\pi$-metrizable if and only if $X$ is almost
$\sigma$-paracompact and locally $\pi$-metrizable.
\end{theorem}

\begin{proof}
Obviously, we only need to show the sufficiency.

Let $X$ be almost $\sigma$-paracompact and locally $\pi$-metrizable.
For each $x\in X$, there exists an open neighborhood $V_{x}$ of $x$
such that $V_{x}$ is $\pi$-metrizable. Then $\{V_{x}: x\in X\}$ is
an open covering for $X$. Since $X$ is almost $\sigma$-paracompact,
there exists a $\sigma$-locally finite open collection
$\mathcal{V}$ refining $\{V_{x}: x\in X\}$ and $\cup\mathcal{V}$ is
dense in $X$. We denote $\mathcal{V}$ by $\mathcal{V}=\bigcup_{m\in
\mathbb{N}}\mathcal{V}_{m}$. By the regularity, we can assume that
$\overline{\mathcal{V}}=\{\overline{V}: V\in \mathcal{V}\}$ refines
$\{V_{x}: x\in X\}$. Obviously, $\overline{\mathcal{V}}$ is
$\sigma$-locally finite. Fix an $m\in \mathbb{N}$. For each
$\overline{V}\in \overline{\mathcal{V}_{m}}$, since
$\pi$-metrizability is preserved by the closure of open subspaces,
$\overline{V}$ is $\pi$-metrizable, and therefore, let
$\mathcal{P}(\overline{V})=\bigcup_{n\in
\mathbb{N}}\mathcal{P}_{mn}(\overline{V})$ be a $\sigma$-discrete
$\pi$-base for $\overline{V}$, where $\mathcal{P}_{n}(\overline{V})$
is discrete in $\overline{V}$ for each $n\in \mathbb{N}$. In fact,
for each $\overline{V}\in \overline{\mathcal{V}}$ and
$W\in\mathcal{P}(\overline{V})$, we can also assume that $W\subset
V$. Put
$\mathcal{P}_{mn}=\bigcup_{\overline{V}\in\overline{\mathcal{V}}}\mathcal{P}_{mn}(\overline{V})$
and $\mathcal{P}=\bigcup_{m, n\in\mathbb{N}}\mathcal{P}_{mn}$. Then
$\mathcal{P}$ is a $\sigma$-locally finite $\pi$-base for $X$.
Firstly, $\mathcal{P}$ is a $\pi$-base for $X$. In fact, let $U$ be
a nonempty open subset for $X$. Since $\cup\mathcal{V}$ is dense in
$X$, there is an $V\in \mathcal{V}$ such that $U\cap
V\neq\emptyset$. It follows from $W\subset V$ for each
$W\in\mathcal{P}(\overline{V})$ that there exists a
$W\in\mathcal{P}(\overline{V})$ such that $W\subset U\cap V$. Now,
we show that $\mathcal{P}_{mn}$ is locally finite for each $m, n\in
\mathbb{N}$. For each $x\in X$, since $\overline{\mathcal{V}_{m}}$
is locally finite, there exists an open neighborhood $U(x)$ of $x$
such that $U(x)$ intersects only finitely many elements of
$\overline{\mathcal{V}_{m}}$, We denote those finitely many elements
by $\overline{V}_{1}, \cdots, \overline{V}_{k}$. Then we need only
to find an open neighborhood $G$ of $x$ such that $G$ intersects
only finitely many elements of
$\bigcup_{i=1}^{k}\mathcal{P}_{mn}(\overline{V_{i}})$. Clearly,
$\mathcal{P}_{mn}(\overline{V_{i}})$ is locally finite in $X$ for
each $1\leq i\leq k$. For each $1\leq i\leq k$, there exists an open
subset $U_{i}$ with $x\in U_{i}$ such that $U_{i}$ intersects only
finitely many elements of $\mathcal{P}_{mn}(\overline{V_{i}})$. Let
$G=U(x)\cap (\bigcap_{i=1}^{k}U_{i})$. Clearly, $G$ is an open
neighborhood of $x$ and intersects only finitely many elements of
$\bigcup_{i=1}^{k}\mathcal{P}_{mn}(\overline{V}_{i})$.
\end{proof}

{\bf Remark} (1) We can not omit the condition ``$X$ is almost
$\sigma$-paracompact'' in Theorem~\ref{t10}. Indeed,
Isbell-Mr\'{o}wka space $\psi(D)$~\cite{MS} is locally
$\pi$-metrizable and non-$\pi$-metrizable, where $D$ is a discrete
space with $|D|=\aleph_{1}$. However, it is easy to see that
$\psi(D)$ is not an almost $\sigma$-paracompact space.

(2) We can not replace ``$X$ is almost $\sigma$-paracompact'' by
``$X$ is almost paracompact'' in Theorem~\ref{t10}, where a space is
called {\it almost paracompact}~\cite{sm} if, for each open covering
$\mathcal{U}$, there is a locally finite open collection
$\mathcal{V}$ such that $\mathcal{V}$ refines $\mathcal{U}$ and
$\cup\mathcal{V}$ is dense in $X$. In fact, Isbell-Mr\'{o}wka space
$\psi(\mathbb{\mathbb{N}})$~\cite{MS} is a $\pi$-metrizable space
and non-almost paracompact.

Next, we discuss the second-countable $\pi$-metrizable spaces.

It is clear that second-countable $\pi$-metrizability is preserved
by open subspaces, closures of open subspaces, and dense subspaces.
As countability, let $X$ be a $\pi$-metrizable space. Then $X$ is a
second-countable $\pi$-metrizable space if $X$ satisfies one of the
following conditions:
\begin{enumerate}
\item $X$ is separable;

\item $X$ is Lindel$\ddot{\mbox{o}}$f;

\item $X$ is pseudocompact.
\end{enumerate}

{\bf Remark} It is well known, for a metrizable space $X$, that $X$
is separable if and only if $X$ is Lindel$\ddot{\mbox{o}}$f.
However, there is a separable and $\pi$-metrizable space, which is
not a Lindel$\ddot{\mbox{o}}$f space, for example, Isbell-Mr\'{o}wka
space $\psi(\mathbb{\mathbb{N}})$~\cite{MS}.

The following result is easy to see.

\begin{proposition}
Second-countable $\pi$-metrizability is preserved by quasi-open
maps.
\end{proposition}

However, there exists a non-$\pi$-metrizable space, which is the
image of a second-countable $\pi$-metrizable space under a closed
and at most two-to-one map, see Example~\ref{e0}.

\begin{theorem}\label{t1}
A space $Y$ is the image of a second-countable $\pi$-metrizable
space $X$ under a closed and at most two-to-one map if and only if
$Y$ is separable.
\end{theorem}

\begin{proof}
Necessity. Since a second-countable $\pi$-metrizable space is
separable, $Y$ is separable.

Sufficiency. If $Y$ is finite, it is obvious. Hence we can assume
that $Y$ is infinite. Let $\{d_{n}: n\in \mathbb{N}\}$ be a
countable dense subset for $Y$, where $d_{n}\neq d_{m}$ for distinct
$n, m\in \mathbb{N}$. Let $X=\{(n, d_{n}): n\in
\mathbb{N}\}\cup (\{p\}\times Y)$ and endow $X$ with the subspace
topology of $\mathbb{N}_{\ast}\times Y$, where
$\mathbb{N}_{\ast}=\mathbb{N}\cup\{p\}$ is the Alexandroff
compactification of $\mathbb{N}$.

Claim: $X$ is second-countable $\pi$-metrizable.

Let $\mathcal{P}_{n}=\{(n, d_{n})\}$ and $\mathcal{B}_{n}=\{(p,
d_{n}): \{d_{n}\}\in\tau (Y)\}$ for each $n\in \mathbb{N}$.
Obviously, $\mathcal{P}_{n}$ and $\mathcal{B}_{n}$ are discrete for
each $n\in \mathbb{N}$, where $\mathcal{B}_{n}=\emptyset$ if
$\{d_{n}\}\not\in\tau (Y)$. Then $\bigcup_{n\in
\mathbb{N}}(\mathcal{P}_{n}\cup\mathcal{B}_{n})$ is a $\pi$-base for
$X$. Indeed, let $O$ be a nonempty open subset of $X$. Then there
exist an $m\in \mathbb{N}$ and an open subset $U$ of $Y$ such that
$O=((\mathbb{N}_{\ast}\setminus\{1, 2, \cdots, m-1\})\times U)\cap
X$. Obviously, we only need to prove that $O\cap\{(n, d_{n}): n\in
\mathbb{N}\}\neq\emptyset$ or $O\cap L\neq\emptyset$, where $L=\{(p,
d_{n}): \{d_{n}\}\in \tau (Y)\}$. If $O\cap (\{p\}\times
Y)=\emptyset$, then it is obvious. Therefore, we can assume that
$O\cap (\{p\}\times Y)\neq\emptyset$. Suppose that $O\cap\{(n,
d_{n}): n\in \mathbb{N}\}=\emptyset$. Then $O\subset \{p\}\times Y$.
Since $U$ is open in $Y$, there exists an $n\in \mathbb{N}$ such
that $d_{n}\in U$. Assume that $O\cap L=\emptyset$. Then $(n,
d_{n})\in O$ if $n\geq m$, this is a contradiction. Hence $n< m$.
Since $U\setminus\{d_{1}, d_{2}, \cdots, d_{m-1}\}\neq\emptyset$,
there is an $n_{0}\geq m$ such that $d_{n_{0}}\in U$. Therefore,
$(n_{0}, d_{n_{0}})\in O$, this is a contradiction. Hence $O\cap
L\neq\emptyset$. Then there exists a $k\in \mathbb{N}$ such that
$(p, d_{k})\in\mathcal{B}_{k}$ and $(p, d_{k})\in O\cap L$.

Let $f:X\rightarrow Y$ be the natural projection map. Since
$\mathbb{N}_{\ast}$ is compact, the projection of
$\mathbb{N}_{\ast}\times Y$ onto $Y$ is a closed map. It follows
from $X$ is a closed subspace of $\mathbb{N}_{\ast}\times Y$ that
$f$ is a closed map. Obviously, for each $y\in Y$, $f^{-1}(y)$ is at
most two points set. Hence $f$ is a closed and at most two-to-one
map.
\end{proof}

\begin{corollary}
A space $Y$ is the image of a second-countable $\pi$-metrizable
space $X$ if and only if $Y$ is separable.
\end{corollary}

\begin{example}\label{e0}
There exists a regular and separable space $X$, which is not a
$\pi$-metrizable space. Therefore, closed and at most two-to-one
maps don't preserve $\pi$-metrizability by Theorem~\ref{t1}.
\end{example}

\begin{proof}
Suppose that $\mathbb{I}=[0, 1]$ is the closed unit interval with a
subspace of the usual topology $\mathbb{R}$, and
$X=\mathbb{I}^{\mathbb{I}}$ with the product topology. Then $X$ is a
regular and separable space. However, $X$ is not a $\pi$-metrizable
space by \cite[Theorem 3.11]{sd}.
\end{proof}

\medskip
\section{Strongly $\pi$-metrizable spaces}
\begin{definition}
Let $\mathcal{P}$ be a collection of open subsets of $X$.
$\mathcal{P}$ is called a {\it strong $\pi$-base}~\cite{AV} if
$\mathcal{P}=\bigcup_{x\in X}\mathcal{P}_{x}$ and, for each $x\in
X$, $\mathcal{P}_{x}$ is a strong $\pi$-base at point $x$, that is,
$\mathcal{P}_{x}$ is a $\pi$-base at point $x$ and every open
neighborhood of $x$ contains all but finitely many elements of
$\mathcal{P}_{x}$.

$X$ is called {\it strongly $\pi$-metrizable} if $X$ has a
$\sigma$-discrete strong $\pi$-base. $X$ is called {\it
second-countable strongly $\pi$-metrizable} if $X$ has a countably
strong $\pi$-base.
\end{definition}

It is obvious that every metrizable space is strongly
$\pi$-metrizable, and every strongly $\pi$-metrizable space is
$\pi$-metrizable. The implications of the converses are not true.

(1) Isbell-Mr\'{o}wka space $\psi(\mathbb{N})$~\cite{MS} is a
strongly $\pi$-metrizable space, but it is not a metrizable space;

(2) Let $K$ be a discrete space with $|K|=\aleph_{1}$.
$K^{\aleph_{1}}$ is $\pi$-metrizable by \cite{sd}, and however,
$K^{\aleph_{1}}$ is a non-strongly $\pi$-metrizable space by the
following Theorem~\ref{t5}.

Clearly, if $\mathcal{P}=\bigcup_{x\in X}\mathcal{P}_{x}$ is a
strong $\pi$-base for $X$, then every infinite subfamily of
$\mathcal{P}_{x}$ is a strong $\pi$-base at point $x$. Therefore, we
have the following result.

\begin{proposition}\label{t3}
If $X$ has a strong $\pi$-base, then every point of $X$ has a
countably strong $\pi$-base.
\end{proposition}

In \cite{sd}, D. Stover given a non-metrizable topological group,
which is $\pi$-metrizable. However, we have the following result by
Theorem~\ref{t3}.

\begin{theorem}\label{t6}
If $X$ is a topological group with a strong $\pi$-base, then $X$ is
metrizable.
\end{theorem}

\begin{proof}
Obviously, $X$ is has a countable $\pi$-character by
Proposition~\ref{t3}. Then $X$ is metrizable by \cite[Theorems 3.6
and 1.8]{CW}.
\end{proof}

\begin{theorem}\label{t2}
For a topological space $X$, the following are equivalent:
\begin{enumerate}
\item $X$ is a
strongly $\pi$-metrizable space;
\item $X$ has a $\sigma$-HCP strong $\pi$-base;
\item $X$ has a $\sigma$-locally finite strong $\pi$-base.
\end{enumerate}
\end{theorem}

\begin{proof}
(1)$\Rightarrow$(2). It is obvious. From \cite[Lemma 2.1]{sd}, it is
easy to see that (3)$\Rightarrow$(1).

It is easy to see that (2)$\Rightarrow$(3) by the proof of
(2)$\Rightarrow$(3) in Theorem~\ref{t0}.
\end{proof}

It is obvious that strongly $\pi$-metrizability is preserved by open
subspaces or dense subspaces. However, we have the following
questions.

\begin{question}
Is strongly $\pi$-metrizability preserved by the closures of open
subspaces?
\end{question}

\begin{question}
Let $X$ be a paracompact space. If $X$ is locally strongly
$\pi$-metrizable, then is $X$ strongly $\pi$-metrizable?
\end{question}

\begin{theorem}\label{t4}
Quasi-open and closed maps preserve strongly $\pi$-metrizability.
\end{theorem}

\begin{proof}
Let $f:X\rightarrow Y$ be an open and closed map, where $X$ is a
strongly $\pi$-metrizable space. It follows from Theorem~\ref{t2}
that $X$ has a $\sigma$-HCP strong $\pi$-base $\mathcal{P}$. Since
closed maps preserve HCP collections, $f(\mathcal{P})$ is a
$\sigma$-HCP collection of subsets of $Y$. Since $f$ is a quasi-open
map, $\{\mbox{int}f(P): P\in \mathcal{P}\}$ is a $\sigma$-HCP strong
$\pi$-base of $Y$. In fact, for each $y\in Y$, choose a fixed point
$x_{y}\in f^{-1}(y)$. Then $\{\mbox{int}f(P):
P\in\mathcal{P}_{x_{y}}\}$ is a strong $\pi$-base at point $y$.
Hence $Y$ is a strongly $\pi$-metrizable space by Theorem~\ref{t2}.
\end{proof}

\begin{corollary}\label{c4}
Open and closed maps preserve strongly $\pi$-metrizability.
\end{corollary}

\begin{corollary}
Irreducible closed maps preserve strongly $\pi$-metrizability.
\end{corollary}

\begin{proof}
It is easy to see by Theorem~\ref{t4} and the proof of
Corollary~\ref{c3}.
\end{proof}

\begin{example}
There exists a non-strongly $\pi$-metrizable space $X$, which is the
inverse image of a strongly $\pi$-metrizable space under
a perfect map.
\end{example}

\begin{proof}
Let $D$ be an uncountable set and endow $D$ with a discrete
topology. Let $Z$ be the Alexandroff compactification of $D$, that
is, $Z=D\cup\{z\}$. Let $X=\psi(\mathbb{N})\times Z$ be the product
topology, where $\psi(\mathbb{N})$ is Isbell-Mr\'{o}wka space. Then
the projection $\pi_{1}: X\rightarrow \psi(\mathbb{N})$ is a perfect
map. However, $X$ is a non-strongly $\pi$-metrizable space. Suppose
not, then the point $x=(1, z)\in X$ has a countable strong
$\pi$-base $\mathcal{P}_{x}$ by Proposition~\ref{t3}. Since every
open neighborhood of $z$ in $Z$ has the form $Z-A$ with $A$ a finite
subset of $D$, there exists a countable subset $L\subset D$ such
that $(D-L)\subset \pi_{2}(P)$ for each $P\in\mathcal{P}_{x}$.
Choose a point $y\in D-L$. Then $\{1\}\times (Z-\{y\})$ is an open
neighborhood of $(1, z)$. But $P\not\subset \{1\}\times (Z-\{y\})$
for each $P\in\mathcal{P}_{x}$, this is a contradiction.
\end{proof}

\begin{question}
Do irreducible perfect maps inversely preserve strongly
$\pi$-metrizability?
\end{question}

\begin{theorem}\label{t7}
Let $Y$ be a Fr$\acute{e}$chet space. Then $Y$ is the image of a
strongly $\pi$-metrizable space $X$ under a perfect map if and only
if $Y$ is strongly $d$-separable.
\end{theorem}

\begin{proof}
Necessity. It is obvious.

Sufficiency. Let $\bigcup_{n\in \mathbb{N}}D_{n}$ be a dense subset
for $Y$, where $D_{n}$ is a closed and discrete subspace of $Y$ for
each $n\in \mathbb{N}$. Put $E_{n}=\bigcup_{i=1}^{n}D_{i}$ for each
$n\in \mathbb{N}$. Obviously, for each $n\in \mathbb{N}$, $E_{n}$ is
a closed and discrete subspace of $Y$.

By the same notations in Theorem~\ref{t1}. Let $X=(\cup\{\{n\}\times
E_{n}: n\in \mathbb{N}\})\cup (\{p\}\times Y)$. Then $X$ is a
strongly $\pi$-metrizable space. Indeed, let $\mathcal{B}_{n}=\{(n,
d): d\in E_{n}\}$ for each $n\in \mathbb{N}$. Obviously,
$\mathcal{B}_{n}$ is discrete for each $n\in \mathbb{N}$. Then
$\mathcal{B}=\bigcup_{n\in \mathbb{N}}\mathcal{B}_{n}$ is a strong
$\pi$-base for $X$.

(i) If $x=(n, d)\in \{n\}\times E_{n}$ for some $n\in \mathbb{N}$,
then let $\mathcal{B}_{x}=\{(n, d)\}$, and therefore,
$\mathcal{B}_{x}$ is a strong $\pi$-base at point $x$.

(ii) If $x=(p, d)\in\{p\}\times \bigcup_{n\in \mathbb{N}}E_{n}$,
then there exists an $m\in \mathbb{N}$ such that $d\in E_{m}$. We
let $\mathcal{B}_{x}=\{(i, d): i\geq m\}\subset \mathcal{B}$. Then
$\mathcal{B}_{x}$ is a strong $\pi$-base at point $x$.

(iii) If $x=(p, d)\in\{p\}\times (Y\setminus\bigcup_{n\in
\mathbb{N}}E_{n})$, then $d\in \overline{\bigcup_{n\in
\mathbb{N}}E_{n}}$. Since $Y$ is Fr$\acute{e}$chet, there exists a
sequence $\{d_{n}\}_{n=1}^{\infty}$ in $\bigcup_{n\in \mathbb{N}}E_{n}$ such that
$d_{n}\rightarrow d$ as $n\rightarrow\infty$. By the induction on
$\mathbb{N}$, we can define an increasing sequence
$\{m_{d_{n}}\}_{n=1}^{\infty}$ in $\mathbb{N}$ such that, for
each $n\in\mathbb{N}$, $m_{d_{n}}>n$, and $d_{n}\in E_{m_{d_{n}}}$.
Let $\mathcal{B}_{x}=\{(m_{d_{n}}, d_{n}): n\in\mathbb{N}\}$. Then
$\mathcal{B}_{x}$ is a strong $\pi$-base at point $x$. In fact, let
$O$ be an open neighborhood at point $x$. Then there exist an
$\mathbb{N}_{m}=\mathbb{N}_{\ast}\setminus\{1, 2, \cdots, m-1\}$ and
an open neighborhood $U$ at point $d$ in $Y$ such that
$(\mathbb{N}_{m}\times U)\cap X\subset O$. Since $d_{n}\rightarrow
d$, there is a $l\in \mathbb{N}$ such that $\{d_{n}: n\geq
l\}\subset U$. Put $k=\mbox{max}\{l, m\}$. Then, for each $n\geq k$,
we have $(m_{d_{n}}, d_{n})\in (\mathbb{N}_{m}\times U)\cap X$.

Let $f:X\rightarrow Y$ be the natural projection map. Since
$\mathbb{N}_{\ast}$ is compact, the projection of
$\mathbb{N}_{\ast}\times Y$ onto $Y$ is a closed map. It follows
from $X$ is a closed subspace of $\mathbb{N}_{\ast}\times Y$ that
$f$ is a closed map. For each $y\in Y$, since $f^{-1}(y)$ is
homeomorphic to a subspace of $\mathbb{N}_{\ast}$ containing the
limit point $p$, $f^{-1}(y)$ is compact. Hence $f$ is a perfect map.
\end{proof}

\begin{corollary}
Let $Y$ be a Fr$\acute{e}$chet space. Then $Y$ is the image of a
strongly $\pi$-metrizable space $X$ under a closed map if and only
if $Y$ is strongly $d$-separable.
\end{corollary}

\begin{theorem}\label{t8}
Let $Y$ be a Fr$\acute{e}$chet space. Then $Y$ is the image of a
second-countable strongly $\pi$-metrizable space $X$ under a closed
and at most two-to-one map if and only if $Y$ is separable.
\end{theorem}

\begin{proof}
By the same notations in Theorem~\ref{t1}. Let $X=\{(n, d_{n}):
n\in \mathbb{N}\}\cup (\{p\}\times Y)$, $\mathcal{P}_{n}=\{(n,
d_{n})\}$ and $\mathcal{B}_{n}=\{(p, d_{n}): \{d_{n}\}\in\tau (Y)\}$
for each $n\in \mathbb{N}$. By a similar argument of
Theorem~\ref{t7}, we can show that $\bigcup_{n\in
\mathbb{N}}(\mathcal{P}_{n}\cup\mathcal{B}_{n})$ is a countably
strong $\pi$-base for $X$ and $Y$ is the image of $X$ under a closed
and at most two-to-one map.
\end{proof}

\begin{example}\label{e1}
There exists a Fr$\acute{e}$chet, $\pi$-metrizable, separable,
regular, and non-strongly $\pi$-metrizable space $X$. Therefore,
closed and at most two-to-one maps don't preserve strongly
$\pi$-metrizability by Theorem~\ref{t8}.
\end{example}

\begin{proof}
Let $X$ be the sequence fan space $S_{\omega}$, which is obtained
from the topological sum of $\omega$ many copies of the convergent
sequence by identifying all the limit points to a point. Then $X$ is
Fr$\acute{e}$chet, $\pi$-metrizable, regular, and separable. Let
$X=\{x_{ni}: i, n\in \mathbb{N}\}\cup \{a\}$, where
$x_{ni}\rightarrow a$ as $i\rightarrow\infty$ for each $n\in
\mathbb{N}$. However, $X$ is non-strongly $\pi$-metrizable. Suppose
not, there exists a collection $\mathcal{P}_{a}$ of open subsets of
$X$ such that $\mathcal{P}_{a}$ is a strong $\pi$-base at point $a$.
By an induction on $\mathbb{N}$, we can choose an increasing
sequence $\{n_{k}\}_{k}\subset \mathbb{N}$ and a subfamily
$\{P_{n_{k}}: k\in \mathbb{N}\}$ of $\mathcal{P}_{a}$ such that, for
each $k\in \mathbb{N}$, $P_{n_{k}}\cap\{x_{n_{k}i}: i\in
\mathbb{N}\}\neq\emptyset$ and $P_{n_{k+1}}\in
\mathcal{P}_{a}\setminus\{P_{n_{i}}: i\leq k\}$, where
$P_{n_{k}}\in\mathcal{P}_{a}$ for each $k\in \mathbb{N}$. Choose a
point $x_{n_{k}i_{n_{k}}}\in P_{n_{k}}\cap\{x_{n_{k}i}: i\in
\mathbb{N}\}$ for each $k\in \mathbb{N}$. Then $$U=\{x_{n_{k}i}:
i>i_{n_{k}}, k\in \mathbb{N}\}\cup\{x_{ni}: n\in
(\mathbb{N}\setminus\{n_{k}: k\in \mathbb{N}\}),
i\in\mathbb{N}\}\cup\{a\}$$ is an open neighborhood of $a$. But
$P_{n_{k}}\not\subset U$ for each $k\in \mathbb{N}$, this is a
contradiction with $\mathcal{P}_{a}$ is a strong $\pi$-base at point
$a$.
\end{proof}

\begin{theorem}\label{t9}
If $X_{n}$ is strongly $\pi$-metrizable for each $n\in \mathbb{N}$,
then $X=\prod_{n\in \mathbb{N}}X_{n}$ is strongly $\pi$-metrizable.
\end{theorem}

\begin{proof}
It follows from Proposition~\ref{t3} that every point of $X_{n}$ has
a countably strong $\pi$-base for each $n\in \mathbb{N}$. For each
$n\in \mathbb{N}$, let $\mathcal{P}^{n}=\bigcup_{x(n)\in
X_{n}}\mathcal{P}_{x(n)}^{n}$ be a $\sigma$-discrete strong
$\pi$-base for $X_{n}$, where
$\mathcal{P}_{x(n)}^{n}=\{U_{x(n)}^{i}: i\in \mathbb{N}\}$. For
every $x\in X$, put
$$\mathcal{P}_{x}=\{\prod_{k=1}^{n}U_{x(k)}^{n}\times\prod_{k=n+1}^{\infty}X_{k}:
n\in \mathbb{N}\}.$$ Then $\mathcal{P}_{x}$ is a strong $\pi$-base at
point $x$. Indeed, for any $x\in U\in\tau (X)$, then $U$ has the
form $U=\prod_{i=1}^{m}W_{i}\times\prod_{i=m+1}^{\infty}X_{i}$,
where $W_{i}$ is open in $X_{i}$ for each $1\leq i\leq m$. Then for
every $1\leq i\leq m$, there is a $k_{i}\in \mathbb{N}$ such that
$U_{x(i)}^{n}\subset W_{i}$ for every $n\geq k_{i}$. Put
$k_{0}=\mbox{max}\{k_{1}, \cdots, k_{m}, m\}$. Therefore, for every
$n>k_{0}$ and $1\leq i\leq m$, $U_{x(i)}^{n}\subset W_{i}$, and
hence
$\prod_{k=1}^{n}U_{x(k)}^{n}\times\prod_{k=n+1}^{\infty}X_{k}\subset
U$ for every $n>k_{0}$.

Let $\mathcal{P}=\bigcup_{x\in X}\mathcal{P}_{x}$. By the proof of
\cite[Proposition 3.1]{sd}, it is easy to see that $\mathcal{P}$ is
$\sigma$-locally finite. Hence $X$ is strongly $\pi$-metrizable by
Theorem~\ref{t2}.
\end{proof}

\begin{corollary}
If $X_{n}$ has a strong $\pi$-base for each $n\in \mathbb{N}$, then
$X=\prod_{n\in \mathbb{N}}X_{n}$ also has a strong $\pi$-base.
\end{corollary}

\begin{theorem}\label{t5}
Let $\kappa$ be an uncountable cardinal numbers. If $X_{\alpha}$ contains at least two points for each $\alpha
<\kappa$, then the product topology $X=\prod_{\alpha
<\kappa}X_{\alpha}$ does not have a strong $\pi$-base at any point of
$X$. In particular, $X$ is non-strongly $\pi$-metrizable.
\end{theorem}

\begin{proof}
Suppose not; there is a point $x\in X$ such that the point $x$ has a
countably strong $\pi$-base $\mathcal{P}_{x}$. Then there exists a
$\beta <\kappa$ such that $\pi_{\beta}(P)=X_{\beta}$ for each $P\in
\mathcal{P}_{x}$. Since $X_{\beta}$ is at least two points set, we
choose a point $y\in X_{\beta}\setminus\{\pi_{\beta}(x)\}$. Then
$(X_{\beta}\setminus \{y\})\times\prod_{\alpha\in (\kappa
-\{\beta\})}X_{\alpha}$ is an open neighborhood of $x$. However,
$P\not\subset (X_{\beta}\setminus \{y\})\times\prod_{\alpha\in
(\kappa -\{\beta\})}X_{\alpha}$ for each $P\in \mathcal{P}_{x}$,
this is a contradiction.
\end{proof}

\begin{question}
Is it true that for any non-strongly $\pi$-metrizable spaces $X$ and
$Y$, we have that $X\times Y$ is also non-strongly $\pi$-metrizable?
\end{question}

\begin{question}
Does there exist a non-strongly $\pi$-metrizable space $X$ such that
$X^{n}$ is strongly $\pi$-metrizable for some $n\in \mathbb{N}$?
\end{question}

{\bf Acknowledgements}. We wish to thank
the referee for the detailed list of corrections, suggestions to the paper, and all her/his efforts
in order to improve the paper. Moreover, we would like to thank the
Professor Kedian Li for his valuable suggestions.

\bibliographystyle{amsplain}

\end{document}